\documentclass[letterpaper, 10 pt, conference]{ieeeconf}  % Comment this line out if you need a4paper
\usepackage{algorithm, algorithmic}
\usepackage{booktabs} % for professional tables
\usepackage{graphicx,graphics,amsthm,amssymb,enumerate,tikz}
\usepackage{mathtools}
\usepackage{arydshln}
\usetikzlibrary{positioning,calc}
\usepackage{fancyhdr}

\theoremstyle{plain}
\newtheorem{theorem}{Theorem}
\newtheorem{proposition}{Proposition}
\newtheorem{lemma}{Lemma}

\theoremstyle{definition}
\newtheorem{definition}{Definition}
\newtheorem{assumption}{Assumption}

\theoremstyle{remark}
\newtheorem{remark}{Remark}

\def\R{\mathbb{R}}
\def\z{\mathbf{z}}

\def\Q{\mathbf{Q}}
\def\P{\mathbf{P}}

\IEEEoverridecommandlockouts                              % This command is only needed if 
\title{\LARGE \bf
Input-to-State Stability Framework for Fully Distributed Primal–Dual Dynamics for Quadratic GNEPs Without Multiplier Consensus
}

\author{Shao-An Yin
\thanks{* This work was performed at the University of Minnesota with support from NSF ECPN Award 2311007. The author is an independent researcher.}%
}\fancypagestyle{firstpage}{
    \fancyhf{}

    \fancyfoot[C]{\footnotesize
\copyright~2026 IEEE. Personal use of this material is permitted.
Permission from IEEE must be obtained for all other uses, in any current or
future media, including reprinting/republishing this material for advertising
or promotional purposes, creating new collective works, for resale or
redistribution to servers or lists, or reuse of any copyrighted component
of this work in other works.}
}

\begin{document}

\maketitle
\thispagestyle{empty}
\pagestyle{empty}
\thispagestyle{firstpage}

%%%%%%%%%%%%%%%%%%%%%%%%%%%%%%%%%%%%%%%%%%%%%%%%%%%%%%%%%%%%%%%%%%%%%%%%%%%%%%%%
\begin{abstract}
Generalized Nash Equilibrium Problems (GNEPs) often arise in multi-agent engineering applications that require distributed algorithms. Unlike traditional approaches that enforce consensus on multipliers, our method removes the need to share multipliers, reducing communication and improving privacy. As a result, different initializations can lead to different GNEs, including non-variational ones. We establish convergence under sufficient conditions using an input-to-state stability (ISS) framework.
\end{abstract}

%%%%%%%%%%%%%%%%%%%%%%%%%%%%%%%%%%%%%%%%%%%%%%%%%%%%%%%%%%%%%%%%%%%%%%%%%%%%%%%%
\section{Introduction}
Generalized Nash Equilibrium Problems (GNEPs) are noncooperative games where agents’ feasible sets are coupled by shared constraints, arising in applications such as power systems, traffic, and communication networks. For example, in traffic routing, each vehicle minimizes travel time subject to shared road capacity \cite{maiorano_dynamics_2000, zhou_generalized_2005, meanfield, MPG, NEURIPS2021_174a61b0, tsaknakis_minimax_2023}, while in resource allocation and power grid management, agents maximize utility under shared capacity limits \cite{contreras_numerical_2004, le_cleach_algames_2020, Shen2023, ma_decentralized_2013}. Since these problems are large-scale and information is often distributed or privacy-sensitive, we focus on scalable distributed solution methods.

In this context, the goal of a distributed algorithm is to compute a Generalized Nash Equilibrium (GNE) without revealing private cost or constraint data. Existing methods require exchanging decision variables and multipliers, which enforces consensus and limits convergence to a variational GNE (v-GNE). In contrast, our approach avoids multiplier exchange, enabling convergence to a broader set of GNEs, including non-variational GNEs, depending on the initialization, while reducing communication and improving privacy. We further demonstrate the method on a robot sensing problem, showing convergence of the fully distributed algorithm.

\subsection{Literature Review}
GNEPs, introduced in \cite{debreu_social_1952, rosen_existence_1965}, model games in which agents' feasible sets depend on others' decisions. Motivated by scalability and privacy concerns, research has increasingly focused on distributed algorithms. Most distributed methods target v-GNEs of strongly monotone games and rely on multiplier consensus \cite{facchinei_generalized_2007, paccagnan_distributed_2016, yu_distributed_2017, yi_pavel_2019, bianchi_continuous-time_2021}. Existing approaches include stochastic robustness analysis \cite{yu_distributed_2017}, operator splitting \cite{yi_pavel_2019, huang_distributed_2021}, and continuous-time dynamics \cite{bianchi_continuous-time_2021}, typically requiring auxiliary variables and additional communication for multiplier consensus.

Unlike existing distributed methods, ours targets a broader class of GNEs. A preliminary version with shared equality constraints and convex local constraints appeared in \cite{yin_acc2026}. This work extends it to linear inequalities, where the induced primal--dual dynamics need not be monotone and the solution set may be disconnected, making the theoretical convergence analysis substantially more challenging. We address this using bounded-input bounded-output and bounded-trajectory analysis, while retaining convergence, under sufficient conditions, to a broader set of GNEs, including non-variational ones, depending on initialization.

%%%%%%%%%%%%%%%%%%%%%%%%%%%%%%%%%%%%%%%%%%%%%%%%%%%%%%%%%%%%%%%%%%%%%%%%%%%%%%%%
\section{Preliminaries}
\subsection{Notation}
Boldface $\mathbf{1}$ and $\mathbf{0}$ denote all-ones and all-zeros vectors of appropriate dimensions. For a vector $a \in \R^q$, $a^\top$ denotes the transpose, $a^\top b$ the inner product with $b \in \R^q$, and $a \le b$ an element-wise comparison. For a matrix $A \in \R^{p \times q}$, $A^\top$ denotes its transpose, and $\bar{\sigma}(A)$ denotes its largest singular value. Following \cite{cherukuri_asymptotic_2016, ebrahimi_robust_2019}, the projection operator for $a, b \in \R$ is given by:
\begin{equation}\label{equ:proj_oper}
\begin{split}
        [a]^+_b &= \begin{cases}
        a, & \text{if }b > 0 \\
        \max \{ 0, a\} & \text{if }b = 0 
    \end{cases}= \begin{cases}
            a, & \text{if }b>0 \, \lor \, a \ge 0\\
            0, & \text{if }b=0, \, a<0
        \end{cases}.
\end{split}
\end{equation}
To simplify the notation, for $a, b \in \mathbb{R}^q$, let $[a]^+_b$ denote a vector whose $i$-th component is $[a_i]^+_{b_i}$. 

\subsection{Problem Formulation}

A Generalized Nash Equilibrium Problem (GNEP) involves $N$ agents. Each agent $v$ controls decision variables $z_v \in \R^{n_v}$. Let $\z \in \R^n$ denote the vector formed by stacking all agents' decision variables, where $n = \sum_{v=1}^N n_v$. We use $\z_{-v}$ to represent the vector of all decision variables except those of agent $v$. Each agent has a cost function $f_v(z_v, \z_{-v}): \R^{n} \to \R$, which depends on its own decision variables $z_v$ as well as those of the other agents $\z_{-v}$. In a GNEP, the feasible decision set of each agent may also depend on the decisions of the other agents. More precisely, each agent $v$ aims to solve the following optimization problem:
\begin{equation}\label{equ:GNEP}
    \begin{split}
        \begin{matrix*}[l]
            \min_{z_v} & f_v(z_v, \z_{-v})\\
            \text{subject to} &  g_{v}(z_v, \z_{-v}) \le  \mathbf{0}_{m_v}\\
        \end{matrix*}
    \end{split}.
\end{equation}
Specifically, the individual cost function $f_v(z_v, \z_{-v})$ is quadratic:
\begin{equation*}
\begin{split}
f_v(z_v, \z_{-v})
= z_v^\top \Q_v z_v
+ z_v^\top \P_v \z_{-v}
+ \mathbf{r}_v^\top z_v,
\end{split}
\end{equation*}
where $\Q_v \in \mathbb{R}^{n_v \times n_v}$, $\P_v \in \mathbb{R}^{n_v \times (n-n_v)}$, and $\mathbf{r}_v \in \mathbb{R}^{n_v}$ are defined for each agent $v$.

In this work, we focus on a class of GNEPs in which the agents are jointly constrained by a global constraint set
\begin{equation*}
\Omega = \left \{ \z \, | \, g(\z) \le \mathbf{0}_p \right \}.
\end{equation*}
However, each agent is not assumed to know the complete global constraint set. Instead, agent $v$ only senses the subset of the global constraints relevant to it, represented by
\begin{equation*}
g_v(z_v, \z_{-v}) \le \mathbf{0}_{m_v}.
\end{equation*}
Thus, $g_v$ denotes the portion of the global constraint $g$ available to agent $v$, rather than an additional set of constraints. The subsets sensed by different agents may overlap, and therefore
\begin{equation*}
    m = \sum_{v=1}^N m_v \ge p.
\end{equation*}
Local constraints are included as a special case when a sensed constraint depends only on $z_v$.

We further assume that the global constraints are linear, so that
\begin{equation*}
g(\z) = A \, \z - \mathbf{b},
\end{equation*}
where $A$ is a matrix of size $p \times n$ and $\mathbf{b}$ is a vector of size $p$. Hence, each row of $A \, \z-\mathbf{b}\le\mathbf{0}_p$
represents one of the global constraints, while each agent only senses the rows relevant to it.

Thus, for each agent $v$, given the decisions of all other agents $\bar{\z}_{-v}$, the constraints sensed by agent $v$ can be written as
\begin{align*}
    g_v(z_v, \bar{\z}_{-v}) = &A(m_v, n_v) \, z_v +  \sum_{j \in [N], j \neq v} A(m_v, n_j) \, \bar{z}_{j}\\ &- \mathbf{b}(m_v),
\end{align*}
where $A(m_v, n_j)$ represents the $m_v$ rows of $A$ sensed by agent $v$ and the $n_j$ columns associated with $z_j \in \R^{n_j}$, and $\mathbf{b}(m_v)$ represents the corresponding $m_v$ entries of $\mathbf{b}$. Therefore, $g_v$ is the subset of the global linear constraints sensed by agent $v$. These constraints may depend on the decisions of multiple agents, while a local constraint appears as a special case when the corresponding constraint depends only on $z_v$.

We define the pseudo-gradient $\nabla F: \R^n \to \R^n$, where the $v$-th block is the gradient of agent $v$'s cost function with respect to its own decision variables, given by
\begin{equation*}
\begin{split}
        &\nabla F(\z) :=  \begin{bmatrix*}
        \nabla_{z_1} f_{1}(z_1, \z_{-1}) \\ \nabla_{z_2} f_{2}(z_2, \z_{-2}) \\ \vdots\\ \nabla_{z_N}f_{N}(z_N, \z_{-N})
    \end{bmatrix*} \\ =  &\begin{bmatrix*}
        (\Q_1 + \Q_1^\top) \, z_1 + \P_1 \, \z_{-1} + \mathbf{r}_1 \\
        (\Q_2 + \Q_2^\top) \, z_2 + \P_2 \, \z_{-2} + \mathbf{r}_2 \\
        \vdots\\
        (\Q_N + \Q_N^\top) \, z_N + \P_N \, \z_{-N} + \mathbf{r}_N \\
    \end{bmatrix*} := F \, \z + \mathbf{r},
\end{split}
\end{equation*}
where $F$ is an $n \times n$ square matrix, $\mathbf{r}$ is a vector of size $n$.
\begin{definition}[Strong Monotone Game]
The GNEP \eqref{equ:GNEP} is strongly monotone if there exists $\delta>0$ such that
\begin{equation*}
    x^\top F x\ge\delta\|x\|^2,\qquad \forall \, x\in\mathbb{R}^n.
\end{equation*}
\end{definition}

\begin{assumption}\label{ass:stronglymonotone}
    The GNEP (\ref{equ:GNEP}) is strongly monotone.
\end{assumption}
Under Assumption~\ref{ass:stronglymonotone} and the nonemptiness of the feasible set, the existence of at least one GNE is guaranteed \cite{glynn_robinson_facchinei_2004}, even though the set of equilibria may be disconnected.

\subsection{Generalized Nash Equilibrium and KKT Conditions}

The solution concept for the GNEP (\ref{equ:GNEP}) is called Generalized Nash Equilibrium (GNE):
\begin{definition}[Generalized Nash Equilibrium (GNE)]\label{def:GNE}
A feasible point $\z^\star$ is a GNE if, for every agent $v$,
\begin{equation*}
    \begin{split}
    &f_v(z_v^\star, \z_{-v}^\star) \le f_v(z_v, \z_{-v}^\star), \\
    &\forall \, z_v \in \left \{z_v \, | \, g_v(z_v, \z_{-v}^\star) \le \mathbf{0}_{m_v}\right \}.
    \end{split}
\end{equation*}
\end{definition}

\begin{assumption}\label{ass:standard}
Two standard assumptions on the game:
    \begin{enumerate}[(i)]
        \item for every player $v$, the function $f_v(z_v, \z_{-v})$ is convex in $z_{v}$, i.e., $(\Q_v+\Q_v^\top) \succeq 0$ for all $v$.
        \item for every player $v$, a suitable constraint qualification and the Linear Independence Constraint Qualification (LICQ) hold; moreover, $\operatorname{rank}(A)=p$.
    \end{enumerate}
\end{assumption}
\begin{definition}[KKT System]\label{ass:KKT}
    Let $\z_{-v}^\star$ be given, and suppose that $z_v^\star$ is a GNE of the GNEP (\ref{equ:GNEP}), then for all $v$, $\lambda_{v}^\star \in \R^{m_v}$ exist such that
    \begin{equation}\label{equ:kkt_gnep}
    \begin{split}
        &L_v(\z,\,  \lambda_{v}): = f_v(z_v, \z_{-v}) + g_{v}(\z)^\top \lambda_{v}\\
        &\nabla_{z_v}  L_v(\z^\star,\,  \lambda_{v}^\star) = \mathbf{0}_{n_v} \\
        &\mathbf{0}_{m_v} \le \lambda_v^\star \, \bot \, g_{v}(z_v^\star, \, {\z}_{-v}^\star)  \le \mathbf{0}_{m_v}\\
    \end{split}.
    \end{equation}
\end{definition}
Under convexity and LICQ, $\z^\star$ is a GNE if and only if there exist KKT multipliers such that the stacked KKT system (\ref{equ:kkt_gnep}) is satisfied \cite[Sec.~4.2]{facchinei_generalized_2010}.

Existing distributed GNE algorithms typically rely on the variational GNE (v-GNE) framework and establish consensus among the players' Lagrange multipliers to ensure convergence to a v-GNE \cite{paccagnan_distributed_2016, yu_distributed_2017, yi_pavel_2019, bianchi_continuous-time_2021, facchinei_generalized_2007}. In contrast, our method does not require multiplier consensus and can converge to a general GNE.

\subsection{Projected Dynamical System}
\begin{definition}[Projected Dynamical System (PDS), \cite{debreu_social_1952}]
    Given $x \in \hat{K}$ and $v \in \R^n$, where $\hat{K} \subseteq \R^n$ is a closed convex set, define the projection of the vector $v$ at $x$ with respect to $\hat{K}$ by  
    \begin{equation*}
        \Pi_{\hat{K}}(x, v) := \lim_{\epsilon \to 0} \frac{P_{\hat{K}}(x + \epsilon \,  v) - x}{\epsilon},
    \end{equation*}
    where 
    \begin{equation*}
        P_{\hat{K}}(x) := \arg \min_{y \in \hat{K}} \| x - y\|.
    \end{equation*}
A projected dynamical system (PDS) is an ordinary differential equation of the form:
    \begin{equation}\label{equ:PDS}
        \dot{x} = \Pi_{\hat{K}}(x, -\hat{F}(x)).
    \end{equation}
\end{definition}
The next proposition establishes trajectory uniqueness.
\begin{proposition}[\cite{zhang_projected_1996} Theorem 2.5]\label{thm:uniq_tra}
    Let $-\hat{F}(x): \R^n \to \R^n$ be Lipschitz on a closed convex polyhedron $\hat{K} \subseteq \R^n$, then for any $x_0 \in \hat{K}$, there exists a unique solution $t \to x(t)$ to the projected dynamical system with $x(0) = x_0$ defined over the domain $[0, \infty)$.
\end{proposition}
Here, the domain $[0,\infty)$ ensures that the trajectory is well defined for arbitrarily large times, allowing asymptotic analysis as $t\to\infty$.

%%%%%%%%%%%%%%%%%%%%%%%%%%%%%%%%%%%%%%%%%%%%%%%%%%%%%%%%%%%%%%%%%%%%%%%%%%%%%%%%

\section{A Fully Distributed Algorithm}
Algorithm~\ref{alg:main}, the contribution of our work, is a fully distributed method that eliminates the need for multiplier consensus. This allows each agent to keep its individual cost function and associated shared constraints private, thereby reducing information sharing and improving privacy.
\begin{algorithm}[H]
\begin{algorithmic}[1]
\STATE {\bfseries Input:} Initial $z_v(0)$, $\lambda_{v}(0) \ge 0$.
\STATE $\dot{z}_v = -\nabla_{z_v} L_v(\z, \, \lambda_v)$.
\STATE $\dot{\lambda}_{v} = [g_v(z_v, \z_{-v})]^+_{\lambda_{v}}$.
\end{algorithmic}
\caption{Distributed dynamics of each agent $v$}
\label{alg:main}	
\end{algorithm}
In Lines~2–3 of Algorithm~\ref{alg:main}, each agent may need to share its decision vector with all other agents in the worst-case scenario.

The proposed method exchanges only decision variables, while keeping each agent's cost function, multipliers, and constraints private. Decision variables are communicated over the cost-function and shared-constraint connectivity graphs, whereas existing v-GNE methods (e.g., \cite{yi_pavel_2019, cenedese_asynchronous_2021, huang_distributed_2021, bianchi_continuous-time_2021}) additionally exchange Lagrange multipliers and auxiliary variables.

To quantify the savings, let $\mathcal{G}_f=(\mathcal{V},\mathcal{E}_f)$ and $\mathcal{G}_g=(\mathcal{V},\mathcal{E}_g)$ denote the cost-function and shared-constraint graphs, respectively, and assume $\mathcal{G}_f=\mathcal{G}_g$, with each exchanged vector in $\mathbb{R}^d$. Our method requires $|\mathcal{E}_f|d$ communicated entries per iteration, whereas consensus-based methods require $3|\mathcal{E}_f|d$ when accounting for decision, multiplier, and auxiliary-variable exchange. Thus, under this communication model, our method reduces the per-iteration communication cost to one-third.

Lemma~\ref{lem:alg_pds} and Proposition~\ref{thm:uniq_tra} guarantee that the trajectories are well defined for all $t\ge0$ and that the equilibria correspond to KKT points of the game, providing the basis for the asymptotic analysis.
\begin{lemma}\label{lem:alg_pds}
Algorithm~\ref{alg:main} implements the PDS in \eqref{equ:PDS} by stacking the dynamics of all agents. Its trajectory is unique and well-defined for all time, and every equilibrium corresponds to a KKT point of \eqref{equ:kkt_gnep}.
\end{lemma}

\subsection{Convergence of the Difference Vector}
For each shared constraint $g^i(\z)$, let $\lambda^i(t)$ be the trajectory with the smallest initial value among $\{\lambda_v^i(t)\}$, chosen as the reference state. There are $p$ such reference states, one per shared constraint. For any other $\lambda_u^i(t)$ under the same constraint, the difference from its reference evolves as:
\begin{equation}\label{equ:delta_inequ}
\begin{split}
    \dot{\Delta}_{u}^i &:= \dot{\lambda}_u^i - \dot{\lambda}^i = [g^i(\z)]_{{\lambda}_u^i}^+ - [g^i(\z)]_{{\lambda}^i}^+ \\
    &=\begin{cases}
    0 & \text{if }\left (\Delta^i_u \ge 0, \lambda^i > 0 \right) \, \lor \, \left (g^i(\z) \ge 0 \right)\\
    g^i(\z) & \text{if }\Delta^i_u > 0,\,  \lambda^i = 0,\, g^i(\z) < 0\\
    0  & \text{if }\Delta^i_u = 0,\,  \lambda^i = 0,\, g^i(\z) < 0
\end{cases}.
\end{split}
\end{equation}

Thus, each $\Delta_u^i(t)$ is nonincreasing and nonnegative. If agent $u$ is the reference for constraint $i$, then $\Delta_u^i(t)\equiv0$. Let $\Delta_{[v]}\in\mathbb{R}^{m_v}_{\ge0}$ collect the multiplier differences associated with agent $v$, and define $m:=\sum_{v=1}^N m_v$.

Combining the dynamics of all agents yields
\begin{equation}\label{equ:whole_system}
\begin{split}
\dot{\z} &= -F\z-\mathbf{r}-A^\top\lambda-D\Delta\\
\dot{\lambda} &= [A\z-\mathbf{b}]_\lambda^+ \\
\dot{\Delta} &\text{ is given componentwise by \eqref{equ:delta_inequ}}
\end{split},
\end{equation}
where
\begin{equation*}
\begin{split}
    & \lambda=\begin{bmatrix}
    \lambda^1 \\ \lambda^2\\ \vdots \\ \lambda^p
\end{bmatrix}, D = \text{diag} \left ( \begin{bmatrix}
            A(m_1, n_1)^\top\\
            A(m_2, n_2)^\top\\
             \vdots \\
            A(m_N, n_N)^\top \\
        \end{bmatrix}\right), 
    \Delta = 
        \begin{bmatrix}
            \Delta_{[1]}\\
            \Delta_{[2]}\\
             \vdots \\
            \Delta_{[N]}\\
        \end{bmatrix}.
\end{split}
\end{equation*}
We first consider the special case $\Delta(0)=\mathbf{0}$, where all copies of each shared-constraint multiplier are initialized identically.
\begin{lemma} \label{lem:delta0} If $\Delta(0) = \mathbf{0}$, the system will converge to an equilibrium point, which is the v-GNE for the GNEP (\ref{equ:GNEP}).
\end{lemma}
\begin{proof}[Proof sketch]
If $\Delta(0)=\mathbf{0}$, then \eqref{equ:delta_inequ} implies $\Delta(t)\equiv\mathbf{0}$ for all $t\ge0$, so Algorithm~\ref{alg:main} reduces to the standard projected primal--dual dynamics with consensus multipliers. The corresponding primal--dual operator is monotone since $F$ is strongly monotone and the cross terms cancel. Let $x=(\z,\lambda)$ and $x^\star=(\z^\star,\lambda^\star)$. Using the Lyapunov function $V(x)=\frac12\|x-x^\star\|^2$, standard projected-dynamical-system arguments yield $\dot V\le0$. LaSalle's invariance principle, together with the strong monotonicity of $F$ and the full row rank of $A$, then implies convergence to the unique equilibrium $x^\star$. Since this equilibrium satisfies the KKT conditions with consensus multipliers, it corresponds to the v-GNE of \eqref{equ:GNEP}.
\end{proof}

\begin{lemma}\label{lem:det_const}
    For every initial point  $y_0 = (\z(0),\, \lambda(0), \, \Delta(0))$, there exists a $\bar{\Delta}_{[y_0]} \ge \mathbf{0}$ such that
    \begin{equation*}
        \lim_{t \to \infty} \Delta^i (t) = \bar{\Delta}^i_{[y_0]}.
    \end{equation*}
\end{lemma} 
\begin{proof}[Proof sketch]
By \eqref{equ:delta_inequ}, each component $\Delta^i(t)$ is nonincreasing and nonnegative. Hence, $\Delta^i(t)$ is bounded below and therefore converges to some limit $\bar{\Delta}^i_{[y_0]}\ge0$.
\end{proof}

Thus, $\Delta(t)\to\bar{\Delta}_{[y_0]}$, where the limit depends on the initial condition $y_0$. Moreover, by \eqref{equ:delta_inequ}, any component $\Delta^i(t)$ that reaches zero remains zero thereafter.

\subsection{Input-to-State Stability with Respect to Active KKT Points}

We now analyze the primal--dual dynamics with respect to a fixed value $\hat{\Delta}$. For a given $\hat{\Delta}\ge0$, we consider a KKT point at which the shared constraints are active, and denote the corresponding primal--dual variables by $\z(\hat{\Delta})$ and $\lambda(\hat{\Delta})$. Not every $\hat{\Delta}$ necessarily admits such an active KKT point. We therefore define the set of admissible shifts as follows:
\begin{definition}[Admissible active-KKT shifts]
\label{def:Dact}
Define the set $\mathcal{D}_{\mathrm{act}} \subseteq \mathbb{R}^m_{\ge 0}$ as the collection of all
$\hat{\Delta}\ge 0$ for which there exists a pair $(\z(\hat{\Delta}),\lambda(\hat{\Delta}))$
satisfying
\begin{equation}\label{eq:active_kkt_map}
\begin{split}
F \z(\hat{\Delta}) + \mathbf{r} + A^\top \lambda(\hat{\Delta}) + D \hat{\Delta} &= 0,\\
A \z(\hat{\Delta}) - \mathbf{b} &= 0,\\
\lambda(\hat{\Delta}) &\ge 0 .
\end{split}
\end{equation}
\end{definition}

\begin{assumption}[Existence of active-KKT points]
\label{ass:exist_delta}
There exists at least one $\hat{\Delta}\in\mathcal{D}_{\mathrm{act}}$ for which an
associated active-constraint KKT point exists.
\end{assumption}

Consider the error dynamics with respect to the active KKT point in \eqref{eq:active_kkt_map} with
$\tilde{\z} = \z - \z(\hat{\Delta})$ and
$\tilde{\lambda} = \lambda - \lambda(\hat{\Delta})$.
Then
\begin{equation}\label{equ:error}
\begin{split}
    \dot{\tilde{\z}} &= - \left [ F \, \tilde{\z} + A^\top \, \tilde{\lambda} + D \, (\Delta(t) - \hat{\Delta}) \right ]\\
    \dot{\tilde{\lambda}} &= [A \, \tilde{\z}]_{\tilde{\lambda} + \lambda(\hat{\Delta})}^+
\end{split}.
\end{equation}

We treat $\Delta(t) - \hat{\Delta}$ as an external input and denote
$u(t,\hat{\Delta}) := \Delta(t) - \hat{\Delta}$. Let $x := (\tilde{\z},\tilde{\lambda})$ denote the error trajectory. We suppress its dependence on $\hat{\Delta}$ when clear from context.

Our goal is to show that, for any $\hat{\Delta} \ge 0$ such that there exist associated $\z(\hat{\Delta})$ and $\lambda(\hat{\Delta})$ satisfying \eqref{eq:active_kkt_map}, the reduced system obtained by treating $\Delta(t)$ as an exogenous input satisfies
\begin{equation*}
\|x(t,\hat{\Delta})\| \le k_1 e^{-\alpha t/2} \|x(0,\hat{\Delta})\| + k_2 \|u(\cdot,\hat{\Delta})\|_\infty,
\end{equation*}
where $k_1>0$, $k_2>0$, and $\alpha>0$ are constants independent of $\hat{\Delta}$. Importantly, this bound holds with respect to any point satisfying \eqref{eq:active_kkt_map}; hence, it provides an error bound relative to any such point, rather than to a uniquely specified reference solution.

We first establish properties of the cross terms under the new dynamics.
\begin{lemma}[Projection inequalities]
Let $q := A \, \tilde{\z}$ and let $\dot{\tilde{\lambda}} = [q]^+_{\tilde{\lambda} + \lambda(\hat{\Delta})}$. Then, for almost all $t$, the following inequalities hold:
\begin{equation}
\tilde{\lambda}^\top \dot{\tilde{\lambda}} \le \tilde{\lambda}^\top q = \tilde{\lambda}^\top A \, \tilde{\z}.
\label{eq:proj_ineq1}
\end{equation}
\begin{equation}
q^\top \dot{\tilde{\lambda}} = \|\dot{\tilde{\lambda}}\|^2 \le \|q\|^2 = \|A \, \tilde{\z}\|^2.
\label{eq:proj_ineq2}
\end{equation}
\end{lemma}
\begin{proof}[Proof sketch]
Let $v:=\tilde{\lambda}+\lambda(\hat{\Delta})$ and $q:=A\,\tilde{\z}$. Since
$\dot{\tilde{\lambda}}=[q]^+_v$, the projection acts componentwise. If $v_i>0$, then $\dot{\tilde{\lambda}}_i=q_i$; if $v_i=0$, then $\dot{\tilde{\lambda}}_i=\max \{q_i,0 \}$ and $\tilde{\lambda}_i=-\lambda_i(\hat{\Delta})\le0$. Hence, componentwise,
$$
\tilde{\lambda}_i\dot{\tilde{\lambda}}_i\le \tilde{\lambda}_i q_i,
\qquad
q_i\dot{\tilde{\lambda}}_i=\dot{\tilde{\lambda}}_i^2\le q_i^2.
$$
Summing over $i$ gives \eqref{eq:proj_ineq1} and \eqref{eq:proj_ineq2}.
\end{proof}

We now choose a Lyapunov function for the error dynamics in (\ref{equ:error}).
Fix $\mu > 0$, to be specified later, and define
\begin{equation}\label{eq:Vdef}
V(\tilde{\z}, \tilde{\lambda}) :=
\frac12 \|\tilde{\z}\|^2 + \frac12 \|\tilde{\lambda}\|^2
+ \mu \, \tilde{\z}^\top A^\top \tilde{\lambda}.
\end{equation}

To ensure that $V$ is positive definite, let
$x := \begin{bmatrix} \tilde{\z} \\ \tilde{\lambda} \end{bmatrix}$ and define
\begin{equation*}
P :=
\begin{bmatrix}
I & \mu A^\top \\
\mu A & I
\end{bmatrix},
\qquad
V = \frac12 x^\top P x.
\end{equation*}

By the Schur complement, the matrix $P$ is positive definite if
\begin{equation}\label{eq:mu_PD}
I - \mu^2 A A^\top \succ 0,
\qquad \text{equivalently,} \qquad
\mu < \frac{1}{\|A\|}.
\end{equation}

Assume that (\ref{eq:mu_PD}) holds. Then there exist positive constants $c_1$ and $c_2$ such that
\begin{equation*}
c_1 \|x\|^2 \le V(x) \le c_2 \|x\|^2,
\end{equation*}
where $c_1 := \frac12 (1 - \mu \|A\|), \, c_2 := \frac12 (1 + \mu \|A\|)$.

%%%%%%%%%%%%%%%%%%%%%%%%%%%%%
%%%%%%%%%%%%%%%%%%%%%%%%%%%%%

\begin{lemma}
\label{lem:eiss_pd}
For any $\hat{\Delta}\in\mathcal{D}_{\mathrm{act}}$ and the
associated $(\z(\hat{\Delta}),\lambda(\hat{\Delta}))$ in \eqref{eq:active_kkt_map}, consider the error system \eqref{equ:error}, with $u(\cdot)\in L_\infty([0,\infty);\mathbb{R}^m)$. Under
Assumptions~\ref{ass:stronglymonotone},~\ref{ass:standard}, and~\ref{ass:exist_delta}, there exist constants
$k_1,k_2,\alpha>0$, independent of $\hat{\Delta}$, such that
\begin{equation*}
    \|x(t)\|
    \le
    k_1 e^{-\alpha t/2}\|x(0)\|+k_2\|u\|_\infty,
    \qquad \forall \,t\ge0.
\end{equation*}
\end{lemma}
\begin{proof}
Define $x:=(\tilde{\z},\tilde{\lambda})$ and use the shifted Lyapunov function in \eqref{eq:Vdef}:
\begin{equation*}
V(\tilde{\z},\tilde{\lambda})
:=
\frac12\|\tilde{\z}\|_2^2+\frac12\|\tilde{\lambda}\|_2^2+\mu\,\tilde{\z}^\top A^\top\tilde{\lambda}.
\end{equation*}
If $\mu<1/\|A\|$, then $P=\begin{bmatrix}I&\mu A^\top\\ \mu A&I\end{bmatrix}\succ 0$ and
\begin{equation}\label{eq:V_bounds_shift_user}
c_1\|x\|_2^2\le V(x)\le c_2\|x\|_2^2,
\end{equation}
where $c_1:=\tfrac12(1-\mu\|A\|)$, $c_2:=\tfrac12(1+\mu\|A\|)$.

\smallskip
\noindent
Differentiate $V$:
\begin{equation}\label{eq:Vdot_start_shift_user}
\dot V
=
\tilde{\z}^\top\dot{\tilde{\z}}
+\tilde{\lambda}^\top\dot{\tilde{\lambda}}
+\mu\Big(\dot{\tilde{\z}}^\top A^\top\tilde{\lambda}+\tilde{\z}^\top A^\top\dot{\tilde{\lambda}}\Big).
\end{equation}
Substitute the error dynamics $\dot{\tilde{\z}}=-(F\tilde{\z}+A^\top\tilde{\lambda}+Du)$:
\begin{equation}\label{eq:z_dot_term_shift_user}
\tilde{\z}^\top\dot{\tilde{\z}}
=
-\tilde{\z}^\top F\tilde{\z}
-\tilde{\z}^\top A^\top\tilde{\lambda}
-\tilde{\z}^\top Du.
\end{equation}
Using the projection inequality \eqref{eq:proj_ineq1}, we have
\begin{equation}\label{eq:lambda_dot_term_shift_user}
\tilde{\lambda}^\top\dot{\tilde{\lambda}}
\le
\tilde{\lambda}^\top A\tilde{\z}.
\end{equation}
Hence the $A$--cross term cancels:
\begin{equation}\label{eq:cancel_shift_user}
-\tilde{\z}^\top A^\top\tilde{\lambda}+\tilde{\lambda}^\top\dot{\tilde{\lambda}}
\stackrel{a}{\le}
-\tilde{\z}^\top A^\top\tilde{\lambda}+\tilde{\lambda}^\top A\tilde{\z}
=0,
\end{equation}
where step~(a) uses \eqref{eq:lambda_dot_term_shift_user}.

\smallskip
\noindent
Next expand the two $\mu$-terms. First,
\begin{equation}\label{eq:mu_term1_shift_user}
\begin{split}
\dot{\tilde{\z}}^\top A^\top\tilde{\lambda}
&=-(F\tilde{\z}+A^\top\tilde{\lambda}+Du)^\top A^\top\tilde{\lambda}\\
&=
-\tilde{\z}^\top F^\top A^\top\tilde{\lambda}
-\tilde{\lambda}^\top AA^\top\tilde{\lambda}
-u^\top D^\top A^\top\tilde{\lambda}.
\end{split}
\end{equation}
Second, by \eqref{eq:proj_ineq2},
\begin{equation}\label{eq:mu_term2_shift_user}
\tilde{\z}^\top A^\top\dot{\tilde{\lambda}}
=(A\tilde{\z})^\top\dot{\tilde{\lambda}}
\le \|A\tilde{\z}\|_2^2
\le \|A\|^2\|\tilde{\z}\|_2^2.
\end{equation}

\smallskip
\noindent
Plug \eqref{eq:z_dot_term_shift_user}, \eqref{eq:cancel_shift_user},
\eqref{eq:mu_term1_shift_user}, and \eqref{eq:mu_term2_shift_user}
into \eqref{eq:Vdot_start_shift_user} to obtain
\begin{equation}\label{eq:Vdot_prebound_shift_user}
\begin{split}
\dot V
\le\;&
-\tilde{\z}^\top F\tilde{\z}
-\mu\,\tilde{\lambda}^\top AA^\top\tilde{\lambda}
+\mu\|A\|^2\|\tilde{\z}\|_2^2\\
&-\mu\,\tilde{\z}^\top F^\top A^\top\tilde{\lambda}
-\tilde{\z}^\top Du
-\mu\,u^\top D^\top A^\top\tilde{\lambda}.
\end{split}
\end{equation}

\smallskip
\noindent
Since the game is strongly monotone and $A$ has full row rank, let
$\sigma:=\lambda_{\min}(AA^\top)>0$. Then
\begin{equation}\label{eq:quad_bounds_shift_user}
\tilde{\z}^\top F\tilde{\z}\ge \delta\|\tilde{\z}\|_2^2,
\qquad
\tilde{\lambda}^\top AA^\top\tilde{\lambda}\ge \sigma\|\tilde{\lambda}\|_2^2.
\end{equation}
For the cross term, by $\|F^\top A^\top\|\le \|F\|\,\|A\|$ and Young,
\begin{equation}\label{eq:cross_bound_shift_user}
\mu\big\|\tilde{\z}^\top F^\top A^\top\tilde{\lambda}\big\|
\le
\frac{\delta}{4}\|\tilde{\z}\|_2^2
+\frac{\mu^2\|F\|^2\|A\|^2}{\delta}\|\tilde{\lambda}\|_2^2.
\end{equation}
Substitute \eqref{eq:quad_bounds_shift_user} and \eqref{eq:cross_bound_shift_user} into \eqref{eq:Vdot_prebound_shift_user}:
\begin{equation}\label{eq:dissipation_shift_user}
\begin{split}
\dot V
\le\;&
-\Big(\frac{3\delta}{4}-\mu\|A\|^2\Big)\|\tilde{\z}\|_2^2
-\Big(\mu\sigma-\frac{\mu^2\|F\|^2\|A\|^2}{\delta}\Big)\|\tilde{\lambda}\|_2^2\\
&-\tilde{\z}^\top Du
-\mu\,u^\top D^\top A^\top\tilde{\lambda}.
\end{split}
\end{equation}
Define
\begin{equation*}
a_z:=\frac{3\delta}{4}-\mu\|A\|^2,
\qquad
a_\lambda:=\mu\sigma-\frac{\mu^2\|F\|^2\|A\|^2}{\delta}.
\end{equation*}
Choose $\mu>0$ such that
\begin{equation*}
    \mu < \min \left\{
    \frac{1}{\|A\|},
    \frac{3\delta}{4\|A\|^2},
    \frac{\delta \sigma}{2\|F\|^2\|A\|^2}
    \right\}.
\end{equation*}
Then $a_z>0$, $a_\lambda>0$.

\smallskip
\noindent
We next bound the input terms. Apply Young with $a_z$ and $a_\lambda$:
\begin{equation}\label{eq:input1_shift_user}
\|\tilde{\z}^\top D \, u \|
\le
\frac{a_z}{2}\|\tilde{\z}\|_2^2+\frac{\|D\|^2}{2a_z}\|u\|_2^2,
\end{equation}
\begin{equation}\label{eq:input2_shift_user}
\mu \, \|u^\top D^\top A^\top\tilde{\lambda} \|
\le
\frac{a_\lambda}{2}\|\tilde{\lambda}\|_2^2+\frac{\mu^2\|AD\|^2}{2a_\lambda}\|u\|_2^2.
\end{equation}
Combine \eqref{eq:dissipation_shift_user} with \eqref{eq:input1_shift_user}--\eqref{eq:input2_shift_user} and absorb the
$\frac{a_z}{2}\|\tilde{\z}\|^2$ and $\frac{a_\lambda}{2}\|\tilde{\lambda}\|^2$ terms into the dissipation:
\begin{equation*}
\dot V
\le
-\frac{a_z}{2}\|\tilde{\z}\|_2^2
-\frac{a_\lambda}{2}\|\tilde{\lambda}\|_2^2
+\Gamma\|u\|_2^2,
\end{equation*}
where
\begin{equation*}
\Gamma:=\frac{\|D\|^2}{2a_z}+\frac{\mu^2\|AD\|^2}{2a_\lambda}.
\end{equation*}
Let $k:=\frac12\min\{a_z,a_\lambda\}$. Then
\begin{equation*}
\dot V \le -k\|x\|_2^2+\Gamma\|u\|_2^2.
\end{equation*}
Using $V\le c_2\|x\|_2^2$ from \eqref{eq:V_bounds_shift_user}, we obtain $\|x\|_2^2\ge V/c_2$ and hence
\begin{equation}\label{eq:Vdot_alpha_shift_user}
\dot V \le -\alpha V+\Gamma\|u\|_2^2,
\qquad
\alpha:=\frac{k}{c_2},
\end{equation}
almost everywhere.

\smallskip
\noindent
Integrate \eqref{eq:Vdot_alpha_shift_user}:
\begin{equation*}
V(t)
\le
e^{-\alpha t}V(0)+\frac{\Gamma}{\alpha}\|u\|_\infty^2,
\qquad \forall \, t\ge 0.
\end{equation*}
Using \eqref{eq:V_bounds_shift_user}, i.e., $V(t)\ge c_1\|x(t)\|_2^2$ and $V(0)\le c_2\|x(0)\|_2^2$, we obtain
\begin{equation*}
\|x(t)\|_2
\le
\sqrt{\frac{c_2}{c_1}}\,e^{-\alpha t/2}\|x(0)\|_2
+\sqrt{\frac{\Gamma}{\alpha c_1}}\,\|u\|_\infty,
\qquad \forall \, t\ge 0.
\end{equation*}
\end{proof}
%%%%%%%%%%%%%%%%%%%%%%%%%%%%%

%%%%%%%%%%%%%%%%%%%%%%%%%%%%%%%%
With Lemma~\ref{lem:eiss_pd} providing a uniform bound on the trajectory around any active KKT solution as a function of $\hat{\Delta}$, we now state the main theorem via an ISS-like argument.
\begin{theorem}
\label{thm:main}
Let $(\z(t),\lambda(t),\Delta(t))$ be a trajectory of \eqref{equ:whole_system} from any initial condition $y_0$, and let $\bar{\Delta}_{[y_0]}$ be the limit given by Lemma~\ref{lem:det_const}. If
$\bar{\Delta}_{[y_0]}\in\mathcal{D}_{\mathrm{act}}$, then
$$
(\z(t),\lambda(t))
\to
(\z(\bar{\Delta}_{[y_0]}),\lambda(\bar{\Delta}_{[y_0]})),
$$
and $\z(\bar{\Delta}_{[y_0]})$ is the GNE associated with $\bar{\Delta}_{[y_0]}$.
\end{theorem}

\begin{proof}
Let $\bar{\Delta}_{[y_0]}$ be the limit given by Lemma~\ref{lem:det_const}, and assume
$\bar{\Delta}_{[y_0]}\in\mathcal{D}_{\mathrm{act}}$. Choose
$\hat{\Delta}=\bar{\Delta}_{[y_0]}$ and define
\begin{equation*}
u(t,\hat{\Delta}):=\Delta(t)-\bar{\Delta}_{[y_0]}.
\end{equation*}
Then $u(t,\hat{\Delta})\to0$ as $t\to\infty$.

By the ISS-type estimate established in Lemma~\ref{lem:eiss_pd}, for any $\tau\ge0$ and $t\ge\tau$,
\begin{equation*}
\begin{split}
    &\|x(t,\hat{\Delta})\|
    \le
    k_1 e^{-\alpha (t-\tau)/2}\|x(\tau,\hat{\Delta})\|
    +
    k_2\|u(\cdot,\hat{\Delta})\|_{\infty,[\tau,\infty)}.
\end{split}
\end{equation*}
Since $u(t,\hat{\Delta})\to0$, the second term can be made arbitrarily small by choosing $\tau$ sufficiently large, while for such fixed $\tau$ the first term vanishes as $t\to\infty$. Hence,
\begin{equation*}
x(t,\hat{\Delta})\to0,
\end{equation*}
and consequently
\begin{equation*}
    (\z(t),\lambda(t))
    \to
    \left (\z(\bar{\Delta}_{[y_0]}),\lambda(\bar{\Delta}_{[y_0]}) \right).
\end{equation*}
By \eqref{eq:active_kkt_map}, together with the multiplier differences represented by $\bar{\Delta}_{[y_0]}$, the limiting point satisfies the KKT conditions in~\eqref{equ:kkt_gnep}. Hence, $\z(\bar{\Delta}_{[y_0]})$ is the GNE associated with $\bar{\Delta}_{[y_0]}$.
\end{proof}

\begin{remark}
Theorem~\ref{thm:main} provides a sufficient condition for convergence under Assumption~\ref{ass:exist_delta} and $\bar{\Delta}_{[y_0]}\in\mathcal{D}_{\mathrm{act}}$. If $\bar{\Delta}_{[y_0]}\notin\mathcal{D}_{\mathrm{act}}$, the theorem does not establish convergence to a GNE, although the preceding boundedness result remains valid. Removing Assumption~\ref{ass:exist_delta} and relaxing the sufficient condition $\bar{\Delta}_{[y_0]}\in\mathcal{D}_{\mathrm{act}}$ are left for future work.
\end{remark}

%%%%%%%%%%%%%%%%%%%%%%%%%%%%%%%%%%%%%%%%%%%%%%%%%%%%%%%%%%%%%%%%%%%%%%%%%%%%%%%%

\section{Simulation}
In the multi-agent sensing placement problem, each agent $v$ tracks its target $T_v$ while staying close to others for communication. The cost is
\begin{equation*}
    f_v(z_v, \z_{-v}) = \| z_v - T_v\|^2 + \rho_v \sum_{u \neq v} \| z_v - z_u\|^2,
\end{equation*}
with decision variable $z_v \in \R^2$ denoting the position of agent $v$.  
The space is partitioned into half-spaces, and to ensure coverage, the group center must remain on the same side of the boundary as the service region. For group $\mathcal{G}$ in half-space $-n_\mathcal{G}^\top \mathbf{p} - c_\mathcal{G} \le 0$,
\begin{equation*}
    g^\mathcal{G}(\z) = -n_\mathcal{G}^\top \left( \tfrac{1}{|\mathcal{G}|} \sum_{u \in \mathcal{G}} z_u \right) - c_\mathcal{G}.
\end{equation*}
Figure~\ref{fig:setup} shows the outcomes of Algorithm~\ref{alg:main} from five initializations. Ten targets (squares) are assigned to agents grouped into three categories, each constrained to keep its center within a half-plane. The figure highlights the non-uniqueness of the GNE and shows that Algorithm~\ref{alg:main} converges to different equilibria depending on initialization. The source code used to reproduce the numerical experiments is publicly
available in the accompanying repository~\cite{yin_distributedgne_code}.
\begin{figure}[H]
\centering
\includegraphics[width=0.9\columnwidth]{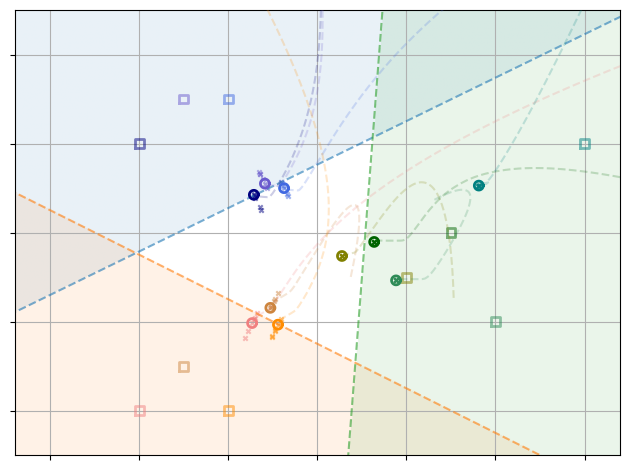} % Reduce the figure size so that it is slightly narrower than the column. Don't use precise values for figure width.This setup will avoid overfull boxes.
    \caption{Sensing placement problem setup and results of Algorithm~\ref{alg:main} under five initial conditions. Colored regions show the half-planes for each group. Dotted lines trace the trajectories from different initializations. Circles mark the converged generalized Nash equilibria (GNE), while crosses ("x") indicate distinct GNE solutions from various starts.}
\label{fig:setup}
\end{figure}
%%%%%%%%%%%%%%%%%%%%%%%%%%%%%%%%%%%%%%%%%%%%%%%%%%%%%%%%%%%%%%%%%%%%%%%%%%%%%%%%

\section{Conclusion}
In this paper, we propose a continuous-time algorithm for quadratic GNEPs that avoids sharing Lagrange multipliers and establish convergence under the sufficient condition
$\bar{\Delta}_{[y_0]} \in \mathcal{D}_{\mathrm{act}}$.
Empirically, the algorithm converges even when this condition is not satisfied. Removing this sufficient condition from the theoretical analysis and extending the framework to general convex cost functions are left for future work.
%%%%%%%%%%%%%%%%%%%%%%%%%%%%%%%%%%%%%%%%%%%%%%%%%%%%%%%%%%%%%%%%%%%%%%%%%%%%%%%%

\addtolength{\textheight}{-12cm}   % This command serves to balance the column lengths

%%%%%%%%%%%%%%%%%%%%%%%%%%%%%%%%%%%%%%%%%%%%%%%%%%%%%%%%%%%%%%%%%%%%%%%%%%%%%%%%

%%%%%%%%%%%%%%%%%%%%%%%%%%%%%%%%%%%%%%%%%%%%%%%%%%%%%%%%%%%%%%%%%%%%%%%%%%%%%%%%

%%%%%%%%%%%%%%%%%%%%%%%%%%%%%%%%%%%%%%%%%%%%%%%%%%%%%%%%%%%%%%%%%%%%%%%%%%%%%%%%

%%%%%%%%%%%%%%%%%%%%%%%%%%%%%%%%%%%%%%%%%%%%%%%%%%%%%%%%%%%%%%%%%%%%%%%%%%%%%%%%

\bibliography{IEEEexample}
\bibliographystyle{IEEEtran}

\end{document}